\documentclass{article}
\usepackage[margin=1in]{geometry}
\usepackage{amsfonts,amsmath,amsthm,color}
\usepackage[all]{xy}

\newcommand{\C}{\mathbb{C}}
\newcommand{\Z}{\mathbb{Z}}
\newcommand{\R}{\mathbb{R}}

\renewcommand{\>}{\right\rangle}
\DeclareMathOperator{\val}{val}

\newcommand{\F}{\mathbb{F}}

\newcommand{\Zt}{\mathbb{Z}^t}

\newcommand{\A}{\mathcal{A}}
\newcommand{\X}{\mathcal{X}}

\newcommand{\vG}{G^\vee}
\newcommand{\cO}{\mathcal{O}}
\newcommand{\cK}{\mathcal{K}}
\newcommand{\Gr}{\mathrm{Gr}}

\newtheorem{thm}{Theorem}
\newtheorem{prop}[thm]{Proposition}

\newtheorem{conj}[thm]{Conjecture}
\newtheorem{definition}[thm]{Definition}
\newtheorem{rmk}[thm]{Remark}
\newtheorem{lem}[thm]{Lemma}

\newtheorem{obs}[thm]{Observation}

\title{Geometry of Positive Configurations in Affine Buildings}
\author{Ian Le and Evan O'Dorney}
\begin{document}
\maketitle

\begin{abstract}

Positive configurations of points in the affine building were introduced in \cite{Le} as the basic object needed to define higher laminations. We start by giving a self-contained, elementary definition of positive configurations of points in the affine building and their basic properties. Then we study the geometry of these configurations. The canonical functions on triples of flags that were defined by Fock and Goncharov in \cite{FG1} have a tropicalization that gives functions on triples of points in the affine Grassmannian. One expects that these functions, though of algebro-geometric origin, have a simple description in terms of the metric structure on the corresponding affine building.

We give a several conjectures describing the tropicalized canonical functions in terms of the geometry of affine buildings, and give proofs of some of them. The statements involve minimal networks and have some resemblance to the max-flow/min-cut theorem, which also plays a role in the proofs in unexpected ways. The conjectures can be reduced to purely algebraic statements about valuations of lattices that we argue are interesting in their own right.

One can view these conjectures as the first examples of intersection pairings between higher laminations. They fit within the framework of the Duality Conjectures of \cite{FG1}.

\end{abstract}

\tableofcontents

\section{Introduction}

This paper has two goals: the first is to give a simple exposition of positive configurations of points in the affine building; the second is to begin the study of the rich geometry of this object. In particular, we lay out several conjectures that give a geometric interpretation of the beautiful set of canonical functions that were defined by Fock and Goncharov in \cite{FG1}. These conjectures serve as the first step towards defining the intersection pairing between higher laminations, which is naturally a pairing between higher laminations for Langlands dual groups.

We start with some motivation. The canonical functions were originally defined on the space of configurations of three principal affine flags associated to the group $G=SL_n(\R)$ (with appropriate adjustments, one can let $G=PGL_n(\R)$ or $GL_n(\R)$). There is a canonical function $f_{ijk}$ for every triple of non-negative integers $i, j, k$ such that $i+j+k=n$.

To a principal affine flag, one can associate a horocycle in the symmetric space $X=G/K$, where $K \subset G$ is the maximal compact subgroup of $G$. One can interpret the value of the canonical functions on a triple of affine flags in terms of the minimal total weighted distance of a spanning network between the three corresponding horocycles.

We are interested in the tropicalization of these functions, explained in \cite{Le} and \cite{GS}. Whereas the canonical functions parameterize configurations of principal affine flags, the tropicalized canonical functions parameterize (positive) virtual configurations of points in the affine building associated to $G$. The adjective ``virtual'' is a technicality that will not be relevant for this paper, while ``positivity'' will only be mentioned in passing. The \emph{affine building} associated to $G$ is the natural tropical analogue of the symmetric space.

These tropicalized functions again should have a geometric interpretation: we expect that they are given by the minimal total weighted distance of a spanning network between three horocycles in the affine building. However, we would rather work with points in the affine building than with horocycles. For this reason, we conjecture that there is a more refined and simpler description of these functions in terms of the geometry of configurations of points in the affine building.

For our purposes, we will restrict our attention to configurations of points in the affine building, which is the most important case. The extension to virtual configurations is straightforward. So the question becomes: given a configuration of three points in the affine building, how do we calculate, in a geometric way, the value of the tropicalized functions $f_{ijk}^t$ on this configuration? We conjecture that they are given by the minimal total weighted distance of a spanning network between the three points in the affine building.

Let us elaborate on the conjecture. We will see below that the affine building has a (non-symmetric) distance function valued in the coweight lattice. If $x_1$ and $x_2$ are points in the affine building, $d(x_1,x_2)$ will be a dominant coweight.

Let $x_1, x_2, x_3$ be any configuration of points in the affine Grassmannian. Let $\omega_i, \omega_j, \omega_k$ be fundamental weights of $SL_n$ with $i+j+k=n$. Then we conjecture that
$$f_{ijk}^t (x_1,x_2,x_3) = \min_{p} \{ \omega_i \cdot d(p,x_1) + \omega_j \cdot d(p,x_2) + \omega_k \cdot d(p,x_3) \},$$
where the minimum is taken over all $p$ in the affine building.

The functions $f_{ijk}^t$ are defined in terms of finding the most negative value of some determinant expression. On the other hand, the conjecture states that the same quantity is computed by minimizing some weighted network. Thus the maximum of one type of quantity is the minimum of another, and the statement resembles max-flow/min-cut.

Below, we prove two special cases of the conjecture, which are in some sense orthogonal:
\begin{enumerate}
\item the configuration of points $x_1, x_2, x_3$ is small in a precise sense described below
\item the points $x_1, x_2, x_3$ all lie in an apartment of the building 
\end{enumerate}
We think that combining the approaches of these two special cases could possibly prove the entire conjecture. It is very interesting to us that our proofs of both cases use max-flow/min-cut, but in each case in a very different way. We think of these conjectures as a kind of max-flow/min-cut for affine buildings.

\section{Background}
\subsection{Affine Grassmannian and affine buildings}

We now begin by laying out the necessary definitions. First we will give an introduction to the affine Grassmannian and the affine building.

Let us define the affine Grassmannian. Let $G$ be a simple, simply-connected complex algebraic group and let $\vG$ be its Langlands dual group. Let $\F$ be a field, which for our purposes will always be $\R$ or $\C$. Let $\cO = \F[[t]]$ be the ring of formal power series over $\F$. It is a valuation ring, where the \emph{valuation} $\val(x)$ of an element $$x = \sum_k a_k t^k \in \F((t))$$ is the minimum $k$ such that $a_k \neq 0$.

Let $\cK = \F((t))$ be the fraction field of $\cO$. Then 
$$\Gr(\F) = \Gr(G) = G(\cK)/G(\cO)$$ is the set of $\F$-points of the \emph{affine Grassmannian} for $G$. It can be viewed as a direct limit of $\F$-varieties of increasing dimension.

For $G=SL_n$, a point in the affine Grassmannian corresponds to a finitely generated, rank $n$, $\cO$-submodule of $\cK^n$ such that if $v_1, \dots, v_n$ are generators for this submodule, then $$v_1 \wedge \dots \wedge v_n=e_1 \wedge \dots \wedge e_n,$$ where $e_1, \dots, e_n$ is the standard basis of $\cK^n$. We will often call such full rank $\cO$-submodules {\it lattices}. $G(\cK)$ acts on the space of lattices with the stabilizer of each lattice being isomorphic to $G(\cO)$, which acts by changing the basis of the submodule while leaving the submodule itself fixed. We will later make use of this interpretation.

The affine Grassmannian $\Gr$ also has a metric valued in dominant coweights: the set of pairs of elements of $\Gr$ up to the action of $G(\cK)$ is exactly the set of double cosets $$G(\cO) \backslash G(\cK) / G(\cO).$$ These double cosets, in turn, are in bijection with the cone $\Lambda_+$ of dominant coweights of $G$. Recall that the coweight lattice $\Lambda$ is defined as $\mathrm{Hom}(\mathbf{G}_m,T)$. The coweight lattice contains dominant coweights, those coweights lying in the dominant cone. For example, for $G=GL_n$, the set of dominant coweights is exactly the set of $\mu=(\mu_1, \dots, \mu_m)$, where $\mu_1 \geq \mu_2 \geq \cdots \geq \mu_n$ and $\mu_i \in \Z$. Let us explain why the set of double cosets is in bijection with the set of dominant coweights.

Given any dominant coweight $\mu$ of $G$, there is an associated point $t^\mu$ in the (real) affine Grassmannian: to a coweight $\mu=(\mu_1, \dots, \mu_m)$ we associate the element of $G(\cK)$ with diagonal entries $t^{\mu_i}$, and then project to the affine Grassmannian. Any two points $p$ and $q$ of the affine Grassmannian can be translated by an element of $G(\cK)$ to $t^0$ and $t^\mu$, respectively, for some unique dominant coweight $\mu$. This gives the identification of the double coset space with $\Lambda_+$.

Under this circumstance, we will write $$d(p,q) = \mu$$ and say that the distance from $p$ to $q$ is $\mu$.

Let us collect some facts about this distance function $d$. Note that this distance function is not symmetric; one can easily check that $$d(p,q)=-w_0 d(p,q)$$ where $w_0$ is the longest element of the Weyl group of $G$ (recall that the Weyl group acts on both the weight space $\Lambda^*$ and its dual $\Lambda$). However, there is a partial order on $\Lambda$ defined by $\lambda > \mu$ if $\lambda - \mu$ is positive (i.e., in the positive span of the positive co-roots). Under this partial ordering, the distance function satisfies a version of the triangle inequality. By construction, the action of $G(\cK)$ on the affine Grassmannian preserves this distance function.

We are interested in the affine Grassmannian, but not in its finer structure as a variety. In fact, we will only consider properties of the affine Grassmannian that depend on the above distance function, and possibly on some positive structure. For this reason, we will introduce affine buildings, a sort of combinatorial skeleton of the affine Grassmannian.

Let us first introduce the affine building for $G=PGL_n$. The affine Grassmannian for $G=PGL_n$ consists of lattices (finitely generated, rank $n$ $\cO$-submodules of $\cK^n$) up to scale: two lattices $L$ and $L'$ are equivalent if $L=cL'$ for some $c \in \C((t))$. The set of vertices of the affine building for $PGL_n$ is precisely given by the points of the affine Grassmannian $\Gr(PGL_n)$.

For any lattices $L_0, L_1, \dots, L_k$, there is a $k$-simplex with vertices at $L_0, L_1, \dots, L_k$ if and only if (replacing each lattice by an equivalent one if necessary) $$L_0 \subset L_0 \subset \cdots \subset L_k \subset t^{-1}L_0.$$
This gives the affine building the structure of a simplicial complex. The affine building for $G=SL_n$ is the same simplicial complex, but where we restrict our attention to those vertices that come from the affine Grassmannian for $G=SL_n$.

The non-symmetric, coweight-valued metric we defined above descends from the affine Grassmannian to the affine building. The notion of a {\emph geodesic} with respect to his metric is sometimes useful. For our purposes, a geodesic in the building is a path that travels along edges in the building from vertex to vertex, such that the sum of the distances from vertex to vertex is minimal (with respect to the partial order defined above). It is a property of affine buildings that geodesics exist. Note that in general there will be many geodesics between two any points.

\subsection{Canonical functions}

We now define the canonical functions of triples of affine flags, which will lead up to the definition of the associated functions on triples of points in the affine Grassmannian.

Let $G=SL_n$ and let $U \subset G$ be the subgroup of unipotent upper triangular matrices. An element of $G/U$ is called a \emph{principal affine flag}. In concrete terms, a principal affine flag is given by a set of $n$ vectors $v_1, \dots, v_n$ where we only care about the forms $$v_1 \wedge \dots \wedge v_k$$ for $k=1, 2, \dots, n-1$. We will require that $$v_1 \wedge \dots \wedge v_n$$ is the standard volume form.

We are interested in the space of (generic) triples of flags up to the left translation action of $G$. Suppose we have three flags $F_1, F_2, F_3$ which are represented by $u_1, \dots, u_n$, $v_1, \dots, v_n$ and $w_1, \dots, w_n$ respectively. Fock and Goncharov define a canonical function $f_{ijk}$ of this triple of flags for every triple of non-negative integers $i, j, k$ such that $i+j+k=n$ and $i, j, k < n$. It is defined by 
$$f_{ijk}(F_1, F_2, F_3)=\det(u_1, u_2, \dots, u_i, v_1, v_2, \dots v_j, w_1, w_2, \dots, w_k),$$ 
and it is $G$-invariant by definition. Note that when one of $i, j, k$ is $0$, these functions only depend on two of the flags. We can call such functions {\em edge} functions, and the remaining functions {\em face} functions.

Given a cyclic configuration of $m$ flags, imagine the flags sitting at the vertices of an $m$-gon, and triangulate the $m$-gon. Then taking the edge and face functions on the edges and faces of this triangulation, we get a set of functions on a cyclic configuration of flags.

\begin{thm} For any triangulation, the edge and face functions form a coordinate chart. Different triangulations yield different functions that are related to the original functions by a positive rational transformation (a transformation involving only addition, multiplication and division) \cite{FG1}.
\end{thm}

We will now analogously define the triple distance functions $f_{ijk}^t$ on a configuration of three points in the affine Grassmannian for $SL_n$. The functions $f_{ijk}^t$ are the same as the functions $H_{ijk}$, which were defined in a slightly different way in \cite{K}. Recall that the affine Grassmannian is given by $G(\cK)/G(\cO)$. For $G=SL_n$, a point in the affine Grassmannian can be thought of as a finitely generated, rank $n$ $\cO$-submodule of $\cK^n$ such that if $v_1, \dots, v_n$ are generators for this submodule, then $$v_1 \wedge \dots \wedge v_n=e_1 \wedge \dots \wedge e_n$$ where $e_1, \dots, e_n$ is the standard basis of $\cK^n$.

Let $x_1, x_2, x_3$ be three points in the affine Grassmannian, thought of as $\cO$-submodules of $\cK^n$. For $i, j, k$ as above, we will consider the quantity
\begin{equation}\label{def fijk}
  f_{ijk}(x_1,x_2,x_3) = -\val(\det(u_1, \dots, u_i, v_1, \dots v_j, w_1, \dots, w_k))
\end{equation}
as $u_1, \dots, u_i$ range over elements of the $\cO$-submodule $x_1$, $v_1, \dots v_j$ range over elements of the $\cO$-submodule $x_2$, and $w_1, \dots, w_k$ range over elements of the $\cO$-submodule $x_3$. Define $f_{ijk}^t (x_1,x_2,x_3)$ as the maximum value of of this quantity, i.e., the largest value of $$-\val(\det(u_1, \dots, u_i, v_1, \dots v_j, w_1, \dots, w_k))$$ as all the vectors $u_1, \dots, u_i,$ $v_1, \dots v_j,$ $w_1, \dots, w_k$ range over elements of the respective $\cO$-submodules $x_1, x_2, x_3$.

There is a more invariant way to define $f_{ijk}^t$. Lift $x_1, x_2, x_3$ to elements $g_1, g_2, g_3$, of $G(\cK)$, then project to three flags $F_1, F_2, F_3 \in G(\cK)/U(\cK)$. Then define $f_{ijk}^t$ to be the maximum of $-\val(f_{ijk}(F_1,F_2,F_3))$ over the different possible lifts followed by projection from $G(\cK)/G(\cO)$ to $G(\cK)/U(\cK)$.

\begin{rmk} It is not hard to check that the edge functions recover the distance between two points in the affine Grassmannian (and hence also the affine building). More precisely, $f_{ij0}^t (x_1,x_2,x_3)$ is given by $\omega_j \cdot d(x_1,x_2)=\omega_i \cdot d(x_2,x_1)$ where $\omega_i$ is a fundamental weight for $SL_n$.
\end{rmk}

\subsection{Positive configurations and conjectures}

Now we may define positive configurations of points in the affine building. 

\begin{definition} Let $x_1, x_2, \dots x_m$ be $m$ points of the real affine Grassmannian. Then $x_1, x_2, \dots x_m$ will be called a positive configuration of points in the affine Grassmannian if and only if there exist ordered bases for $x_i$, 
$$v_{i1}, v_{i2}, \dots, v_{in},$$ such that for each $1 \leq p < q <r \leq m$, and each triple of non-negative integers $i, j, k$ such that $i+j+k=n$,
\begin{itemize}
\item $f_{ijk}^t (x_p,x_q,x_r) = -\val(\det(v_{p1}, \dots, v_{pi}, v_{q1}, \dots v_{qj}, v_{r1}, \dots, v_{rk}))$
\item the leading coefficient of $\det(v_{p1}, \dots, v_{pi}, v_{q1}, \dots v_{qj}, v_{r1}, \dots, v_{rk})$ is positive.
\end{itemize}
\end{definition}

Note that it is important in the above definition that we are taking the valuations of the determinants of the first $i$ (respectively $j, k$) vectors among the bases for $x_p$ (respectively $x_q, x_r$), and not just any $i$ (respectively $j, k$) vectors.

\begin{rmk} By the results of \cite{Le}, it is sufficient to verify the two conditions above for only those triples $p, q, r$ occuring in any one triangulation of the $m$-gon. The valuation condition and the positivity condition for one triangulation implies the these conditions for any other triangulation, and hence for an arbitrary triple $p, q, r$.
\end{rmk}

We can now introduce our conjectures on the tropical functions $f_{ijk}^t$. We need some notation first. Let $\omega_i$ be the $i$-th fundamental weight for $SL_n$: $\omega_i = (1, \dots, 1, 0, \dots, 0)$ where there are $i$ $1$'s and $n-i$ $0$'s. Recall that for any two points $p, q$ in the affine Grassmannian, $d(p,q)$ is an element of the coweight lattice for $SL_n$.

\begin{conj} (Weak form) Let $x_1, x_2, x_3$ be a positive configuration of points in the affine Grassmannian. Then 
$$f_{ijk}^t (x_1,x_2,x_3) = \min_{p} \omega_i \cdot d(p,x_1) + \omega_j \cdot d(p,x_2) + \omega_k \cdot d(p,x_3),$$
where the minimum is taken over all $p$ in the affine Grassmannian.
\end{conj}

There is a stronger, perhaps bolder, form of the conjecture, which is also interesting, although it is less related to the geometry of laminations:

\begin{conj} (Strong form) Let $x_1, x_2, x_3$ be any configuration of points in the affine Grassmannian. Then 
$$f_{ijk}^t (x_1,x_2,x_3) = \min_{p} \omega_i \cdot d(p,x_1) + \omega_j \cdot d(p,x_2) + \omega_k \cdot d(p,x_3),$$
where the minimum is taken over all $p$ in the affine Grassmannian.
\end{conj}

We can make a few elementary observations. 

\begin{obs} On positive configurations of points in the affine building, the functions $f_{ijk}^t$ only depend on metric properties of the configuration within the building \cite{Le}
\end{obs}

Thus our conjecture is giving a more precise description of how the functions $f_{ijk}$ measure metric properties of the configuration.

\begin{obs}\label{one direction} One inequality in the conjecture is clearly true. We have that for any configuration of points $x_1, x_2, x_3$ in the affine Grassmanian, 
$$f_{ijk}^t (x_1,x_2,x_3) \leq \min_{p} \omega_i \cdot d(p,x_1) + \omega_j \cdot d(p,x_2) + \omega_k \cdot d(p,x_3),$$
\end{obs}

This observation follows from the fact that $f_{ijk}^t$ (in fact, $f_{ijk}$) is invariant under the diagonal action of $G(\cK)$ on $x_1, x_2, x_3$, so it is enough to verify the inequality in the case when $p$ is the trivial lattice spanned by $e_1, \dots, e_n$. Thus the conjecture reduces to showing the other inequality.

\begin{obs} If one of $i, j, k$ is equal to $0$, then the conjecture holds. In particular, the conjecture is true for $SL_2$. For example, if $k=0$, 
$$f_{ij0}^t (x_1,x_2,x_3) = \omega_i \cdot d(p,x_1) + \omega_j \cdot d(p,x_2)$$ for any $p$ lying on a geodesic between $x_1$ and $x_2$.
\end{obs}

Let $x_1, x_2, \dots x_m$ be any configuration of points in the affine Grassmannian, with ordered bases for $x_i$, 
$$v_{i1}, v_{i2}, \dots, v_{in},$$ satisfying 
$$f_{ijk}^t (x_p,x_q,x_r) = -\val(\det(v_{p1}, \dots, v_{pi}, v_{q1}, \dots v_{qj}, v_{r1}, \dots, v_{rk})).$$

For any set of dominant coweights $(\lambda_1, \lambda_2, \dots, \lambda_m)$, where 
$$\lambda_i=(\lambda_{i1}, \lambda_{i2}, \dots, \lambda_{in}), \lambda_{i1} \geq \lambda_{i2} \geq \dots \geq \lambda_{in},$$
we can form the configuration of points 
$$x'_1, x'_2, \dots x'_m$$
where $x'_i$ has basis 
$$t^{-\lambda_{i1}}v_{i1}, t^{-\lambda_{i2}}v_{i2}, \dots, t^{-\lambda_{in}}v_{in}$$
(it is easy to check the positivity of the configuration $x'_1, x'_2, \dots x'_m$). We have the following observation \cite{Le}:

\begin{obs} For sufficiently large coweights $(\lambda_1, \lambda_2, \dots, \lambda_m)$, and any $1 \leq p, q, r \leq m$, we have that 
$$f_{ijk}^t (x'_p,x'_q,x'_r) = \min_{p} \omega_i \cdot d(p,x'_1) + \omega_j \cdot d(p,x'_2) + \omega_k \cdot d(p,x'_3).$$
\end{obs}
In other words, the conjecture is true \emph{asymptotically}.

Let us note that the previous observation was originally shown in \cite{Le} for \emph{positive} configurations, but the same proof works for all configurations.

\subsection{Relationship to the Duality Conjectures}

The duality conjectures of Fock and Goncharov posit a relationship between the spaces $\A_{G,S}$ and $\X_{\vG,S}$ where $\vG$ is the Langlands dual group to $G$. In particular, the main part of the conjecture state that there should be a bijection between $\X_{\vG,S}(\Zt)$ (the tropical points of $\X_{G,S}$) and a basis of functions for $\A_{G,S}$. This bijection should satisfy many compatibility relations which we will not discuss here.

This pairing further specializes to a pairing 
$$\X_{\vG,S}(\Zt) \times \A_{G,S}(\Zt) \rightarrow \Z.$$
The pairing works as follows: a point $l \in \X_{\vG,S}(\Zt)$ corresponds to a function $f_l$ on $\A_{G,S}$. A point $l' \in \A_{G,S}(\Zt)$ comes from a taking valuations of some Laurent-series valued point $x_{l'} \in \A_{G,S}(\cK)$. Then the pairing between $l$ and $l'$ is defined by
$$(l,l')=-\val f_l(x_{l'}).$$
The value of $(l,l')$ is independent of the choice of the point $x_{l'}$, as the conjectures state that $f_l$ should be a positive rational function on $\A_{G,S}$ (in fact, it should be a Laurent polynomial in the cluster co-ordinates).

Equivalently, we can describe the pairing as follows. Given a point $l \in \X_{\vG,S}(\Zt)$, take the corresponding function $f_l$ on $\A_{G,S}$. Then if $l' \in \A_{G,S}(\Zt)$, then 
$$(l,l')=f_l^t(l').$$

Points of $\X_{\vG,S}(\Zt)$ and $\A_{G,S}(\Zt)$ correspond to higher laminations for the groups $\vG$ and $G$, respectively. The pairing between $\X_{\vG,S}(\Zt)$ and $\A_{G,S}(\Zt)$ should realize the \emph{intersection pairing} between higher laminations. When $G=SL_2$, this construction reduces to the usual intersection pairing between laminations on a surface \cite{FG1}.

For each $i, j, k$, the function $f_{ijk}$ is a cluster function on $\A_{G,S}$. As such, it should be part of the basis parameterized by $\X_{\vG,S}(\Zt)$. In particular, $f_{ijk}$ is associated to the tropical point of $\X_{\vG,S}(\Zt)$ where the corresponding tropical cluster $x$-variable is $1$ and all other $x$ variables are set to $0$. (For every cluster and cluster variable for $\A_{G,S}$ one can canonically associate a cluster and cluster variable for $\X_{\vG,S}$. This is partly a reflection of the fact that the dual pair of spaces forms a \emph{cluster ensemble}.)

Our conjectures give a way of computing $f_{ijk}(l')$ for $l' \in \A_{G,S}$. Thus they give a geometric interpretation of the intersection pairings. The pairing extends linearly to a pairing between $l \in \X_{\vG,S}(\Zt)$ and $l' \in \A_{G,S}(\Zt)$ whenever $l$ has positive co-ordinates in one of the cluster co-ordinate systems for $\X_{\vG,S}$ associated to a triangulation of $S$ constructed in \cite{FG1}. Thus they give the pairing

$$\X_{\vG,S}(\Zt) \times \A_{G,S}(\Zt) \rightarrow \Z.$$
for any $l' \in \A_{G,S}(\Zt)$ and $l$ contained in a union of open cones inside $\X_{\vG,S}(\Zt)$.

\section{Main theorem}

In this section we prove some partial results towards the strong version of the conjecture.

Because we will be dealing with the case of $G=SL_n, GL_n$ or $PGL_n$, and because we would like to deal with all these cases uniformly we will reformulate the conjectures in terms of lattices. By a \emph{lattice} we mean a $\C[[t]]$-submodule of the vector space $\C((t))^n$ generated by $n$ vectors linearly independent over $\C((t))$.  The simplest lattice is the ``elementary'' lattice $E = \C[[t]]^n$.

Let $i_1,\ldots,i_k$ be nonnegative integers that sum to $n$, and let $L_1, \ldots ,L_k$ be lattices. Define the determinantal valuation
\[
  f^t_{i_1,\ldots,i_k}(L_1,\ldots,L_k) = \max_{v_ij \in L_i} -\val \det(v_{11},\ldots,v_{1i_i},v_{21},\ldots,v_{2i_2},\ldots,v_{k1},\ldots,v_{ki_k}).
\]
We are primarily interested in the cases $k=1$, $2$, and $3$.

The unary determinantal valuation
\[
  f^t_n(L) = \max_{v_j \in L} -\val \det(v_1,\ldots,v_n)
\]
can easily be computed by choosing any $\C[[t]]$-basis $v_1,\ldots,v_n$ for $L$.

As we saw previously, the binary determinantal valuation $v_{ij}(L,M)$ is also well understood:
\begin{prop} Let $L$ and $M$ be lattices. Then there exists $g \in \mathrm{GL}_n(\C((t)))$ and integers $a_1 \geq a_2 \geq \cdots \geq a_n$ such that
\begin{align*}
  gL &= \<e_1, e_2, \ldots, e_n\> \\
  gM &= \<t^{-a_1}e_1, t^{-a_2}e_2, \ldots, t^{-a_n}e_n\>.
\end{align*}
Moreover, $f_{ij}(gL,gM) = a_1 + \ldots + a_j$ and
\[
  f^t_{ij}(L,M) = a_1 + \ldots + a_j + \frac{i f_n^t(L) + j [f_n^t(M) - a_1 - \cdots - a_n]}{n}.
\]
\end{prop}

We see that the numbers $a_1,\ldots,a_n$ are unique up to adding a fixed constant to all of them (though $g$ is far from unique).

We reformulate our main conjecture as saying that the ternary determinantal valuation can be expressed in terms of the binary one.
\begin{conj} \label{conj:main} Let $L$, $M$, $N$ be lattices and $i$, $j$, $k$ nonnegative integers, $i+j+k=n$. Then
\[
  f^t_{ijk}(L,M,N) = \min_{\text{lattices }P} (f^t_{i,j+k}(L,P) + f^t_{j,i+k}(M,P) + f^t_{k,i+j}(N,P) - 2f^t_n(P))
\]
which can also be written as
\[
  f^t_{ijk}(L,M,N) = \min_{\text{lattices }P} (f^t_{ijk}(L,P,P) + f^t_{ijk}(P,M,P) + f^t_{ijk}(P,P,N) - 2f^t_{ijk}(P,P,P))
\]
\end{conj}

The inequality $\leq$ between the two sides is not difficult to prove (Observation~\ref{one direction}); hence the content of the conjecture is the existence of a lattice $P$ for which equality holds.

\subsection{The case of three ``close'' lattices}

Our first partial result is as follows:
\begin{thm} \label{thm:1row}
  Conjecture \ref{conj:main} holds when $E \subseteq L,M,N \subseteq t^{-1}E$. In particular, one of the choices
  \[
    tE, L, M, N, L+M, L+N, M+N, L+M+N
  \]
  works for $P$.
\end{thm}

\begin{proof}
If $E \subseteq L \subseteq t^{-1}E$, then $L$ is determined by its projection $U_1$ to the space $\F^n \equiv t^{-1}E/E$. Likewise $M \simeq E \oplus U_2$ and $N \simeq E \oplus U_3$, where $U_1$, $U_2$, and $U_3$ are subspaces of $V = \F^n$. Now a system of three subspaces of an ambient space
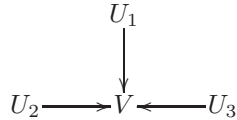
\[
\xymatrix@1{
  & U_1\ar[d] & \\
  U_2\ar[r] & V & U_3\ar[l]
}
\]
forms a representation of the quiver {\tiny $
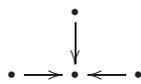
\xymatrix@!0{
  & \bullet\ar[d] & \\
  \bullet\ar[r] & \bullet & \bullet\ar[l]
}$} of Dynkin type $D_4$ and thus can be expressed as a direct sum of its $12$ irreducible representations, of which $3$ are excluded since they correspond to non-injective maps. In the following diagram, we assign names to the remaining $9$ representation types and their basis vectors:
\begin{equation}
\begin{tabular}{c|ccccccccc}
  \multicolumn{1}{c}{Rep.} & $A$ & $A'$ & $A''$ & $B$ & $B'$ & $B''$ & $C$ & $D$ & $S$ \\ \hline
  $U_1$ & $a_\ell$ &&& $b_\ell$ & $b_\ell'$ && $c_\ell$ & $u_\ell$ \\
  $U_2$ && $a_\ell'$ && $b_\ell$ && $b_\ell''$ & $c_\ell$ & $v_\ell$ \\
  $U_3$ &&& $a_\ell''$ && $b_\ell'$ & $b_\ell''$ & $c_\ell$ & $u_\ell + v_\ell$ \\
  $V$ & $a_\ell$ & $a_\ell'$ & $a_\ell''$ & $b_\ell$ & $b_\ell'$ & $b_\ell''$ & $c_\ell$ & $u_\ell, v_\ell$ & $s_\ell$
\end{tabular}
\end{equation}
To compute $f_{ijk}(L,M,N)$, we must search for a choice of vectors $u_1, \dots, u_i \in L$, $v_1, \dots, v_j \in M$, $w_1, \dots, w_k \in N$ minimizing the valuation of the determinant in \eqref{def fijk}. There is no reason not to choose vectors in either $t^{-1}V$ or $V$, and then the valuation of the determinant depends only on the number of $t$'s involved, as long as the $n$ underlying vectors in $V$ are linearly independent. So we have the following interpretation.
\begin{lem} \label{lem:linind}
We have $f_{ijk}(L,M,N) = g$, where $g$ is the maximum number of linearly independent vectors that may be chosen from $U_1$, $U_2$, $U_3$, with the restriction that at most $i$ vectors from $U_1$, $j$ from $U_2$, $k$ from $U_3$ may be chosen.
\end{lem}
Moreover, we may limit our vectors to the bases of $U_1$, $U_2$, $U_3$ constructed above. Now $g$ depends in a purely combinatorial way on $i$, $j$, $k$, and the multiplicities of the irreducible representations.

We will now interpret $g$ as the maximum flow in a certain graph. Draw four layers of vertices as follows:
\begin{itemize}
  \item A single source vertex;
  \item One \emph{representation vertex} for each of the irreducible components of the representation of $D_4$ determined by $U_1$, $U_2$, $U_3$;
  \item Three \emph{$U$-vertices}, labeled $U_1$, $U_2$, and $U_3$;
  \item A single sink vertex.
\end{itemize}
Then draw arrows from each level to the next as follows:
\begin{itemize}
  \item Each representation vertex is joined to the source with an arrow whose capacity is the dimension of the portion of $V$ that it corresponds to (always $1$ except for the representation $D$, where it is $2$);
  \item Each representation vertex is joined to each $U$-vertex $U_m$ with an arrow whose capacity is the dimension of the portion of $U_m$ that it corresponds to (always $0$ or $1$); for future reference, the arrows with capacity $1$ are labeled with the appropriate basis vector of $U_i$.
  \item The three $U$-vertices $U_1$, $U_2$, and $U_3$ are joined to the sink with arrows of capacity $i$, $j$, and $k$, respectively.
\end{itemize}
Now it is readily verified that any choice of basis vectors satisfying the conditions of Lemma \ref{lem:linind} can be represented by a flow in this graph, where for each vector chosen, one unit of fluid flows from the source through the appropriate representation-to-$U$ edge and then out the sink. Hence the maximum flow is $g$.
We turn our attention to the cuts of the constructed graph. Note that once each of the $U_m$ is either cut from or ``soldered'' to the sink (the vertices may be identified once the decision is made not to cut the edge), the graph becomes a union of noninteracting subgraphs, one for each representation vertex. So the minimal cut in each of the eight cases is easily determined:
\begin{enumerate}
\item If all three $U_m$ are cut, no further cuts are necessary, and we obtain a cut of capacity $i+j+k=n$.
\item If two of the $U_m$ are cut, say $U_2$ and $U_3$, then the representation vertices with connections to $U_1$ (namely, those of type $A$, $B$, $B'$, $C$, and $D$) can be dealt with by cutting these connections, which are labeled by the basis vectors of $U_1$. Hence there are cuts of capacity $j+k+\dim U_1$ and, symmetrically, $i+k+\dim U_2$ and $i+j+\dim U_3$.
\item If only one $U_m$ is cut, say $U_3$, then we must further cut one unit for each representation vertex of type $A$, $A'$, $B$, $B'$, $B''$, or $C$, and two units for each representation vertex of type $D$. This amounts to one unit for each vector in a basis of $U_1 + U_2$. So we get cuts of capacity $k + \dim(U_1 + U_2)$ and, symmetrically, $j + \dim(U_1 + U_3)$ and $i + \dim(U_1 + U_2)$.
\item Finally, if none of the $U_m$ are cut, then there is no better option than cutting the inflow to each representation vertex, excepting those of type $S$, for which of course no cut is necessary. So we get a cut of capacity $\dim(U_1 + U_2 + U_3)$.
\end{enumerate}
So we have reached a second checkpoint in the computation of $f_{ijk}$:
\begin{lem}\label{lem:min}
$f_{ijk}(L,M,N) = \min\{i+j+k,j+k+\dim U_1, i+k+\dim U_2, i+j+\dim U_3, k + \dim(U_1 + U_2), j + \dim(U_1 + U_3), i + \dim(U_1 + U_2), \dim(U_1 + U_2 + U_3)\}$.
\end{lem}

Finally, we must relate the eight terms on the right-hand side of this lemma to the right-hand side of Conjecture \ref{conj:main} for the lattices $P$ listed in Theorem \ref{thm:1row}. The key is that in max-flow/min-cut, there is always a flow that uses every edge of the minimal cut at full capacity. We will use this flow to bound the $f_{\bullet,\bullet}$ terms from above, proving that the right-hand side is at most $f_{ijk}(L,M,N)$, since as previously remarked the reverse inequality is trivial (Observation \ref{one direction}).

\begin{enumerate}
\item If $i+j+k$ is the minimal cut, then there are $i$ vectors from $U_1$, $j$ vectors from $U_2$, and $k$ vectors from $U_3$, all linearly independent. Picking $P = E$, we find $f_{i,j+k}(L,E) \leq i$ by taking the $i$ linearly independent vectors from $L$ and filling out with vectors from $V$. Calculating the other two terms by symmetry, we get
\begin{align*}
  &f_{i,j+k}(L,tE) + f_{j,i+k}(M,tE) + f_{k,i+j}(N,tE) - 2f_n(tE) \\
  &\leq i + j + k  \\
  &= f_{i,j,k}(L,M,N).  
\end{align*}
\item If $j+k+\dim U_1$ is the minimal cut, then there are $\dim U_1 \leq i$ vectors from $U_1$, $j$ vectors from $U_2$, and $k$ vectors from $U_3$, all linearly independent. We will pick $P = L$. The most difficult term is $f_{j,i+k}(M,L)$ which can be bounded by picking the $j$ vectors from $U_2$ and the basis for $U_1$, filling out with vectors from $V$, getting a bound of $ j + \dim{U_1}$. The other terms are either symmetric or trivial, and we get
\begin{align*}
  &f_{i,j+k}(L,L) + f_{j,i+k}(M,L) + f_{k,i+j}(N,L) - 2f_n(L) \\
  &\leq  \dim U_1 + j+\dim U_1) + k + \dim U_1) - 2(\dim U_1) \\
  &= j + k + \dim U_1 = f_{i,j,k}(L,M,N).  
\end{align*}
\item If $k + \dim(U_1 + U_2)$ is the minimal cut, there is a basis for $U_1 + U_2$ consisting of at most $i$ vectors from $U_1$ and at most $j$ vectors from $U_2$, and also $k$ vectors chosen from $U_3$ that are linearly independent from $U_1 + U_2$. We pick $P = L+M$. For $f_{i,j+k}(L,L+M)$, it is possible to pick all the vectors in the basis of $U_1 + U_2$ (at most $i$ from $L$ and $j$ from $M$) before filling out with $V$. The term $f_{j,i+k}(M,L+M)$ is symmetric, while in $f_{k,i+j}(N,L+M)$ we can get the $k$ linearly independent vectors in $U_3$ as well as a basis of $L+M$. So in all, we get
\begin{align*}
  &f_{i,j+k}(L,L+M) + f_{j,i+k}(M,L+M) + f_{k,i+j}(N,L+M) - 2f_n(L+M) \\
  &\leq \dim (U_1+U_2) + \dim (U_1+U_2) + k + \dim (U_1+U_2)) - 2(\dim (U_1+U_2)) \\
  &= k + \dim (U_1+U_2) = f_{i,j,k}(L,M,N).
\end{align*}
\item Finally, if $\dim(U_1 + U_2 + U_3)$ is the minimal cut, then $U_1 + U_2 + U_3$ has a basis consisting of at most $i$ vectors from $U_1$, $j$ vectors from $U_2$, and $k$ vectors from $U_3$. This can be used to bound all the terms if we pick $P = L+M+N$:
\begin{align*}
  &f_{i,j+k}(L,L+M+N) + f_{j,i+k}(M,L+M+N) + f_{k,i+j}(N,L+M+N) - 2f_n(L+M+N) \\
  &\leq \dim (U_1+U_2+U_3) + \dim (U_1+U_2+U_3) + \dim (U_1+U_2+U_3) \\
  &\quad - 2(\dim (U_1+U_2+U_3)) \\
  &= \dim (U_1+U_2+U_3) = f_{i,j,k}(L,M,N).
\end{align*}
\end{enumerate}
Of course, the other $4$ possibilities for the minimal cut are symmetric.
\end{proof}

\subsubsection{Connection to Konig's theorem}

The portion of our proof of Theorem \ref{thm:1row} lying between Lemmas \ref{lem:linind} and \ref{lem:min} can be thought of as a combinatorial problem in linear algebra. It can be related to some familiar theorems in combinatorics in a way which we now describe.

\emph{Hall's theorem} or \emph{Hall's Marriage Lemma} is frequently described in terms of the following story: $n$ boys are to be married off to $m$ girls, and each boy-girl pair either likes or dislikes one another. A matching in which all the boys are paired exists if and only if no subset of the boys likes a strictly smaller subset of the girls. Or, in the inanimate language favored by mathematicians:
\begin{thm}[Hall]
If $S_1,S_2,\ldots,S_r$ are sets, then a system of distinct representatives of the $S_i$ (one from each set) exists if and only if for each subset $I \subseteq [r]$,
\[
  \left|\bigcup_{i \in I} S_i\right| \geq |I|.
\]
\end{thm}
A refinement of Hall's theorem is \emph{Konig's theorem.} Instead of giving conditions for a perfect matching to exist, it provides a formula for the maximum number of disjoint pairs that may be made:
\begin{thm}[Konig]
If $S_1,S_2,\ldots,S_r$ are sets, the maximum number of distinct representatives of the $S_i$ (at most one from each set) is
\[
  \min_{I \subseteq [r]} \left( \left| \sum_{i\in I} S_i \right| + r - |I| \right).
\]
\end{thm}

In \cite{M}, Theorem 2, Moshonkin ``linearized'' Hall's theorem in the sense of replacing sets by vector spaces and adjusting conditions accordingly. His result is:
\begin{thm}[Moshonkin] \label{thm:Moshonkin}
If $V_1,V_2,\ldots,V_r$ are subspaces of an ambient vector space $V$, then a system of linearly independent representatives of the $V_i$ exists if and only if for each subset $I \subseteq [r]$,
\[
  \dim \sum_{i \in I} S_i \geq |I|.
\]
\end{thm}

In a similar vein, we would like to state and prove the following linearization of Konig's theorem:
\begin{thm} \label{thm:vsKonig}
If $V_1,\ldots,V_r$ are subspaces of an ambient space $V$, the maximum number of linearly independent representatives from different $V_i$'s is
  \begin{equation}\label{eq:vsKonig}
      \min_{I \subseteq [r]} \left[ \dim \left( \sum_{i\in I} V_i \right) + r - |I| \right].
  \end{equation}
\end{thm}
This result is applicable to our investigations in the following manner: if $U_1$, $U_2$, $U_3$ are subspaces, then the maximum number of linearly independent representatives from different terms of the multiset
\[
  \underbrace{U_1,\ldots,U_1}_{i}, \underbrace{U_2,\ldots,U_2}_{j}, \underbrace{U_3,\ldots,U_3}_{k}
\]
is the quantity $g$ in the statement of Lemma \ref{lem:linind}. On the other hand, the expression \eqref{eq:vsKonig} clearly can only reach its minimal value when the index set $I$ includes either all or none of each of the three strings of repeated $U_i$. Thus Theorem \ref{thm:vsKonig} provides an alternative proof of Lemma \ref{lem:min} from Lemma \ref{lem:linind}.

We now deduce Theorem \ref{thm:vsKonig} from Theorem \ref{thm:Moshonkin}. In fact, the two theorems are readily found to be equivalent.
\begin{proof}[Proof of Theorem \ref{thm:vsKonig}]
Let $M$ be the subset $I \subseteq [r]$ that minimizes \eqref{eq:vsKonig}. Denote 
\begin{align*}
  m &= |M|,\\
  K &= [r]\backslash M,\\
  k &= |K| = r-m,\\
  W &= \sum_{i\in M} V_i,\\
  n &= \dim W.
\end{align*}
We will construct the requisite $n+k$ linearly independent representatives in the following way: we will find a basis for $W$ whose elements come from distinct $V_i$, and we will supplement this basis with vectors in the $k$ spaces $V_i$, $i \in K$, which are all linearly independent in the quotient space $V/W$. The proof will apply the minimality condition to sets $I$ which are respectively subsets and supersets of $M$.

We begin with the second step. Let $\tilde{V}_i$ be the image of $V_i$ in $V/W$. For each $I \subseteq K$, we have the condition
\[
  \dim \left( \sum_{i\in M\cup I} V_i \right) + r - (m + |I|) \geq n + k
\]
which simplifies to
\[
  \dim \sum_{i\in I} \tilde{V}_i \geq |I|.
\]
So the $\tilde{V}_i$, $i \in K$ satisfy precisely the condition of Theorem \ref{thm:Moshonkin} and hence a system of linearly independent representatives exists.

The first step is only slightly trickier. For each $I \subseteq M$, we have
\[
  \dim \left( \sum_{i \in I} V_i \right) + r - |I| \geq n + k
\]
which simplifies to
\[
  \dim \sum_{i\in I} V_i \geq |I| \geq n - m + |I|.
\]
This would be the condition of Theorem \ref{thm:Moshonkin} were it not for the summand $n-m$. So we use a trick. Plugging $I = \emptyset$, we see that $m\geq n$. Let $V' = V \oplus \F^{m-n}$ and $V'_i = V_i \oplus \F^{m-n}$. Then the $V'_i$, $i \in M$, satisfy the conditions of Theorem \ref{thm:Moshonkin} and we can find a basis of $V'$ with one vector from each $V'_i$. Projecting down to $V$, we have a spanning set, of which some $n$ vectors form a basis.
\end{proof}

Both Konig's and Moshonkin's theorems as well as our proof of Theorem \ref{thm:1row} rely on max-flow/min-cut-type results. We believe that this is not accidental, and that in general, defining intersection pairings between higher laminations will involve proving statements about affine buildings that have the flavor of max-flow/min-cut.

\subsection{Apartments}

Let $X = \{x_1, \ldots, x_n\}$ be a basis for $\cK^n$ over $\cK$. For any choice integers $c_1, \dots, c_n$, the set of lattices of the form $t^{c_j} x_j$, $j = 1,2,\ldots,n$ form a subset of the affind building called an apartment.

Our second result is a generalization of our conjecture in the situation where $L_1,\ldots,L_k$ all lie in the same apartment:

\begin{thm}\label{apartments}
If $L_1,\ldots,L_k$ lie in the same apartment and $i_1,\ldots,i_k$ nonnegative integers with sum $n$, then there exists a lattice $P$ such that
\[
  f^t_{i_1,\ldots,i_k}(L_1,\ldots,L_k) = \sum_j f^t_{i_j,n-i_j}(L_j,P) - (n-1)f^t_n(P).
\]
\end{thm}
(\emph{Remark:} We expect that the above generalization holds for general lattices $L_1,\ldots,L_k$.)
\begin{proof}
Use the following combinatorial result:
\begin{thm}[Kuhn and Munkres]
Let $[c_{ij}]$ be a real $n\times n$ matrix. Then the maximal sum of a transversal of $[c_{ij}]$ equals the minimal sum $\sum_i a_i + \sum_j b_j$ where the $a_i$ and $b_j$ satisfy $a_i + b_j \geq c_{ij}$ for all $i,j$. If the $c_{ij}$ are integers, the $a$'s and $b$'s can be taken integral as well.
\end{thm}
In our situation, we can first reduce to the case that $k=n$ and all $i_j$ are $1$. Then write $L_i = \<t^{-c_i1} x_1,\ldots, t^{-c_in} x_n\>$. Using the multilinearity of the determinant, we can assume that the value of $f_{i_1,\ldots,i_k}$ is attained by taking the valuation of the determinant of generators $t^{-c_ij} x_j$ of the respective lattices. Since all the $i$'s must be distinct as must the $j$'s, the valuation of the determinant is a transversal $-\sum_i c_{i\sigma(i)}$ for some $\sigma \in S_n$. Therefore $v_{1\cdots1}(L_1,\ldots,L_n)$ is the maximal transversal sum in the matrix $[c_{ij}]$.

By Kuhn-Munkres, there exist integers $a_i$ and $b_j$ such that $a_i + b_j \geq c_{ij}$ and
\[
  f^t_{1\cdots1}(L_1,\ldots,L_n) = \sum_i a_i + \sum_j b_j.
\]
Choose the lattice
\[
  P = \<t^{-b_1} x_1, \ldots, t^{-b_n} x_n \>.
\]
Obviously, equality $a_i + b_j = c_{ij}$ must hold whenever $c_{ij}$ belongs to the winning transversal. So
\[
  f^t_{1,n-1}(L_i,P) = \max_k \Big(c_{ik} + \sum_{j \neq k} b_j \Big)
  = \sum_j b_j + \max_k(c_{ik} - b_k) = \sum_j b_j + a_i.
\]
Since $f^t_n(P) = \sum_{j} b_j$, we have the desired equality.
\end{proof}

\section{Generalizations}

Theorem~\ref{apartments} hints at some generalizations of conjectures. It is known that the functions $f_{i_1,\ldots,i_k}$ are cluster co-ordinates for $k=2, 3, 4$, so that the corresponding functions $f^t_{i_1,\ldots,i_k}$ also give intersection pairings between higher laminations. In fact, we expect that this is true for all $k$.

Thus, in general, we conjecture that
$$f^t_{i_1,\ldots,i_k}(L_1,\ldots,L_k) = \sum_j f^t_{i_j,n-i_j}(L_i,P) - (k-1)f^t_n(P).$$
Note that if $L_1, \dots L_k$ lie in the same apartment, then Theorem~\ref{apartments} gives the above result. This gives one interpretation of the functions $f^t_{i_1,\ldots,i_k}$ in terms of the geometry of the affine building.

Let us another interpretation when $k=4$. We would like to calculate. We conjecture $f^t_{i_1, i_2, i_3, i_4}(L_1, L_2, L_3, L_4)$ is also given by the minimum of the two following expressions:

$$f^t_{i_4, n-i_4}(L_4, P) + f^t_{i_1, n-i_1}(L_1, P) + f^t_{i_4+i_1, i_2+i_3}(P, Q) + f^t_{i_2, n-i_2}(L_2, Q) + f^t_{i_3, n-i_3}(L_3, Q) - 2f^t_n(P) - 2f^t_n(Q)$$
where $P$ and $Q$ range over all lattices. If $P$ and $Q$ are normalized to have determinant $1$ (possibly by introducing fractional powers of $t$), we get the minimum over $P$ and $Q$ of
$$f^t_{i_1, n-i_1}(L_1, P) + f^t_{i_2, n-i_2}(L_2, P) + f^t_{i_1+i_2, i_3+i_4}(P, Q) + f^t_{i_3, n-i_3}(L_3, Q) + f^t_{i_4, n-i_4}(L_4, Q)$$
and
$$f^t_{i_4, n-i_4}(L_4, P) + f^t_{i_1, n-i_1}(L_1, P) + f^t_{i_4+i_1, i_2+i_3}(P, Q) + f^t_{i_2, n-i_2}(L_2, Q) + f^t_{i_3, n-i_3}(L_3, Q).$$
If $P=Q$, these expressions reduce to the previous expression
$$f^t_{i_1, n-i_1}(L_1, P) + f^t_{i_2, n-i_2}(L_2, P) + f^t_{i_3, n-i_3}(L_3, P) + f^t_{i_4, n-i_4}(L_4, P) - 3f^t_n(P).$$
The content of our conjecture is that allowing $P \neq Q$ does not change the minimum.

Equivalently, it is given by the minimum over $P$ and $Q$ of (again assuming that $P$ and $Q$ are normalized to have determinant $0$)
$$d(P,L_1) \cdot \omega_{i_1} + d(P, L_2) \cdot \omega_{i_2} + d(P, Q)  \cdot \omega_{i_3+i_4}  + d(Q, L_3) \cdot \omega_{i_3}  + d(Q, L_4) \cdot \omega_{i_4}$$
and
$$d(P,L_4) \cdot \omega_{i_4} + d(P, L_1) \cdot \omega_{i_1} + d(P, Q)  \cdot \omega_{i_2+i_3}  + d(Q, L_2) \cdot \omega_{i_2}  + d(Q, L_3) \cdot \omega_{i_3}.$$

Thus we conjecture that $f^t_{i_1, i_2, i_3, i_4}(L_1, L_2, L_3, L_4)$ is calculated by the minimum distance of a weighted network connecting the $L_i$. However, this network can take one of two shapes (three if one counts the degenerate case $P=Q$ separately). One recognizes that these two different networks, along with the weights along the networks, are identical to the spin networks that calculate the untropicalized function $f_{i_1, i_2, i_3, i_4}$.

We believe that in many other more general cases, the functions corresponding to points of $\X_{\vG,S}(\Zt)$ are calculated by some set equivalent spin networks which calculate the same function, and that the associated tropical functions are calculated by the minimal distance over these various weighted networks inside the affine building. For example, when $G=SL_4$ there are two inequivalent spin networks with four leaves with weights $\omega_2, \omega_1, \omega_2, \omega_3$ in that cyclic order. We believe that the corresponding tropical functions are given by the minimums of
$$f^t_{2,2}(L_1, P) + f^t_{1,3}(L_2, P) + f^t_{3, 1}(P, Q) + f^t_{2,2}(L_3, Q) + f^t_{3, 1}(L_4, Q) - 2f^t_n(P) - f^t_n(Q)$$
and
$$f^t_{3,1}(L_4, P) + f^t_{2,2}(L_1, P) + f^t_{1, 3}(P, Q) + f^t_{1,3}(L_2, Q) + f^t_{2, 2}(L_3, Q) - f^t_n(P) - 2f^t_n(Q)$$
respectively. This points towards a general geometric interpretation of intersection pairings between higher laminations.


\bibliographystyle{amsalpha}

\begin{thebibliography}{FG1}


\bibitem[FG1]{FG1} V.V. Fock, A.B. Goncharov. Moduli spaces of local systems and higher Teichmuller theory. Publ. Math. IHES, n. 103 (2006) 1-212. math.AG/0311149.

\bibitem[GS]{GS} A.B. Goncharov, L. Shen. Geometry of canonical bases and mirror symmetry. arXiv:1309.5922

\bibitem[K]{K} J. Kamnitzer. Hives and the fibres of the convolution morphism, Selecta Math. N.S. 13 no. 3 (2007), 483-496.

\bibitem[Le]{Le} I. Le. Higher Laminations and Affine Buildings. arXiv:1209.0812

\bibitem[M]{M} A.G. Moshonkin. Concerning Hall's Theorem, from Mathematics in St. Petersburg, eds. A. A. Bolibruch, A.S. Merkur'ev, N. Yu. Netsvetaev. Amererican Mathematical Society Translations, Series 2, Volume 174, 1996.


\end{thebibliography}

\end{document}